\documentclass[final, 12pt]{article}

\def\showauthornotes{0}
\def\showkeys{0}
\def\showdraftbox{0}

\usepackage{xspace,xcolor,enumerate}
\usepackage{amsmath,amsfonts,amssymb}
\usepackage{color,graphicx}
\usepackage{bbm}

\ifnum\showkeys=1
\usepackage[color]{showkeys}
\fi

\usepackage{dsfont}

\ifnum\showauthornotes=1
\newcommand{\Authornote}[2]{{\sf\small\color{red}{[#1: #2]}}}
\newcommand{\Authorcomment}[2]{{\sf \small\color{gray}{[#1: #2]}}}
\newcommand{\Authorfnote}[2]{\footnote{\color{red}{#1: #2}}}
\else
\newcommand{\Authornote}[2]{}
\newcommand{\Authorcomment}[2]{}
\newcommand{\Authorfnote}[2]{}
\fi

\ifnum\showdraftbox=1
\newcommand{\draftbox}{\begin{center}
  \fbox{%
    \begin{minipage}{2in}%
      \begin{center}%
        \begin{Large}%
          \textsc{Working Draft}%
        \end{Large}\\
        Please do not distribute%
      \end{center}%
    \end{minipage}%
  }%
\end{center}
\vspace{0.2cm}}
\else
\newcommand{\draftbox}{}
\fi

\newtheorem{theorem}{Theorem}[section]

\newtheorem{definition}[theorem]{Definition}
\newtheorem{lemma}[theorem]{Lemma}
\newtheorem{remark}[theorem]{Remark}

\newtheorem{corollary}[theorem]{Corollary}

\def\FullBox{\hbox{\vrule width 6pt height 6pt depth 0pt}}

\def\qed{\ifmmode\qquad\FullBox\else{\unskip\nobreak\hfil
\penalty50\hskip1em\null\nobreak\hfil\FullBox
\parfillskip=0pt\finalhyphendemerits=0\endgraf}\fi}

\def\qedsketch{\ifmmode\Box\else{\unskip\nobreak\hfil
\penalty50\hskip1em\null\nobreak\hfil$\Box$
\parfillskip=0pt\finalhyphendemerits=0\endgraf}\fi}

\newenvironment{proof}{\begin{trivlist} \item {\bf Proof:~~}}
   {\qed\end{trivlist}}

\def\to{\rightarrow}
\def\eps{\varepsilon}
\def\epsilon{\varepsilon}

\def\phi{\varphi}

\newcommand{\defeq}{\stackrel{\mathrm{def}}=}

\renewcommand{\bar}{\overline} 

\newcommand{\ie}{\emph{i.e.},\xspace}

\newcommand{\mcom}{\,,}

\newcommand{\R}{{\mathbb R}}
\newcommand{\E}{{\mathbb E}}

\newcommand{\N}{{\mathbb{N}}}
\newcommand{\Z}{{\mathbb Z}}
\newcommand{\F}{{\mathbb F}}

\newcommand{\B}{\{0,1\}\xspace}

\usepackage{nicefrac}

\let\nfrac=\nicefrac

\newcommand{\abs}[1]{\ensuremath{\left\lvert #1 \right\rvert}}

\newcommand{\Esymb}{\mathbb{E}}
\newcommand{\Psymb}{\mathbb{P}}

\DeclareMathOperator*{\ExpOp}{\Esymb}

\DeclareMathOperator*{\ProbOp}{\Psymb}
\renewcommand{\Pr}{\ProbOp}

\newcommand{\Prob}[2]{\Pr_{{#1}}\left[{#2}\right]}

\newcommand{\Ex}[2]{\ExpOp_{{#1}}\left[{#2}\right]}
\renewcommand{\E}{\ExpOp}

\newfont{\inhead}{eufm10 scaled\magstep1}

\newcommand{\poly}{{\mathrm{poly}}}

\newcommand{\Szemeredi}{Szemer\'edi\xspace}

\allowdisplaybreaks[1]

\usepackage[top=1in,bottom=1in,left=1.25in,right=1.25in]{geometry}
\usepackage{color,graphicx}

\ifnum\showkeys=1
\usepackage[color]{showkeys}
\fi

\definecolor{darkred}{rgb}{0.5,0,0}
\definecolor{darkgreen}{rgb}{0,0.5,0}
\definecolor{darkblue}{rgb}{0,0,0.5}	\usepackage[pdfstartview=FitH,pdfpagemode=None,colorlinks,linkcolor=darkred,filecolor=blue,citecolor=darkred,urlcolor=darkred,pagebackref]{hyperref}

\newcommand{\ent}{\mathsf{Ent}}
\renewcommand{\eps}{\varepsilon}
\renewcommand{\epsilon}{\eps}

\newcommand{\tower}[1]{\mathsf{Tower}({#1})}

\title{%
 An Arithmetic Analogue of Fox's Triangle Removal Argument
}%

\author{%
  Pooya Hatami\thanks{Institute for Advanced Study, Princeton, NJ, USA. This work was done when the author was a student at University of Chicago. email: \texttt{pooyahat@math.ias.edu}. Research supported by the NSF grant No. CCF-1412958.} \and Sushant Sachdeva\thanks{Department of Computer
    Science, Yale University, USA. Part of this work was done when the
    author was a student at Princeton University, and was visiting TTI Chicago. email:
    \texttt{sachdeva@cs.yale.edu}} \and Madhur Tulsiani\thanks{%
    Toyota Technological Institute at Chicago. email: \texttt{madhurt@ttic.edu}. Research supported by NSF
    Career Award CCF-1254044. }%
}

\date{\today}

\begin{document}

\sloppy

\maketitle

\draftbox
\setcounter{page}{1}
\thispagestyle{empty}
\begin{abstract}
We give an arithmetic version of the recent proof of the triangle removal lemma by Fox \cite{Fox11}, for the group $\F_2^n$.

  A triangle in $\F_2^n$ is a triple $(x,y,z)$ such that $x+y+z = 0$. The triangle removal lemma for
  $\F_2^n$ states that for every $\eps  > 0$ there is a $\delta > 0$, such that if a subset $A$ of
  $\F_2^n$  requires the removal of at least $\eps \cdot 2^n$ elements to make it triangle-free,
  then it must contain at least $\delta \cdot 2^{2n}$ triangles. This problem was first studied by
  Green \cite{Green05} who proved a lower bound on $\delta$ using an arithmetic regularity
  lemma. Regularity based lower bounds for triangle removal in graphs were recently improved by Fox
  and we give a direct proof of an analogous improvement for triangle removal in $\F_2^n$.

  The improved lower bound was already known to follow (for
  triangle-removal in all groups), using Fox's removal lemma for
  directed cycles and a reduction by Kr\'{a}l, Serra and Vena
  \cite{KralSV09} (see \cite{Fox11,ConlonF12}). The purpose of this
  note is to provide a direct Fourier-analytic proof for the group
  $\F_2^n.$
\end{abstract}

\section{Introduction}
The triangle removal lemma for graphs states that for every $\eps >
0$, there exists a $\delta > 0$ such that every graph on $n$
vertices with at most $\delta n^3$ triangles can be made triangle-free by deleting less than $\eps n^2$ edges. Contrapositively, this
means that if a graph is at least $\eps$-far from being triangle-free,
\ie one needs to delete more than $\eps n^2$ edges to make it triangle
free, then it must have at least $\delta n^3$ triangles.

The lemma was
originally proved by Ruzsa and \Szemeredi \cite{RuzsaS76} with a bound
of $\delta \geq 1/\tower{\poly(1/\eps)}$, where $\tower{i}$ denotes a
tower of twos of height $i$, \ie $\tower{i} = 2^{\tower{i-1}}$ for $i
\geq 1,$ and $\tower{0}=1$. It took over three decades for this bound to be improved to
$\delta \geq 1/\tower{O(\log(1/\eps))}$ in a remarkable paper of Fox
\cite{Fox11}. 

The above lemma has a direct application to property testing as it states that for
a graph which is $\eps$-far from being triangle-free, a random triple
of vertices is guaranteed to form a triangle with probability at least
$\delta$. Thus if we test $\Omega(1/\delta)$ random triples of
vertices, we will find a triangle with constant probability. This test
can then distinguish such a graph from one which is triangle-free,
since, for a triangle-free graph, the probability of the test
succeeding is 0. The triangle removal lemma and its generalizations to
removal of more general graphs (instead of a triangle) have several
interesting applications in mathematics, and we refer the reader to
the survey \cite{ConlonF12} for a detailed discussion.

An arithmetic version of the triangle removal problem was considered
by Green \cite{Green05}. Let $G$ be an Abelian group with $|G| = N$
and let $A \subseteq G$ be an arbitrary subset. We call a triple
$(x,y,z) \in A^3$ a \emph{triangle} if $x+y+z = 0$. Similar to the
graph case, we say that $A$ is $\eps$-far from being triangle-free if
one needs to remove at least $\eps N$ elements from $A$ to make it
triangle-free. Green proved an arithmetic analogue of \Szemeredi's
regularity lemma \cite{S75}, used in the proof of Ruzsa and
\Szemeredi, and proved a triangle removal lemma for Abelian groups.
\begin{theorem}[Green \cite{Green05}]
For all $\eps \in (0,1]$, there exists a $\delta \geq
1/\tower{\poly(1/\eps)}$ such that for an Abelian group $G$ with $|G|
= N$, if a subset $A \subseteq G$ is $\eps$-far from being
triangle-free, then $A$ must contain at least $\delta N^2$ triangles.
\end{theorem}

Other than being useful for proving the above result, Green's
arithmetic analogue of \Szemeredi's regularity lemma has had many
other applications in combinatorics \cite{geelen2014critical, geelen2014odd} 
and computer science \cite{KhotOD09, Bhat15, Bhat12}. In addition to leading to interesting
analogues of combinatorial statements, the arithmetic setting has at
least one more advantage: the proofs of these statements often proceed
by partitioning the underlying spaces, and because of the structure
provided by subgroups and cosets, the arguments are often cleaner
(specially for vector spaces over finite fields). This makes the
arithmetic setting more attractive in the search for
quantitative improvements to the results.

In this paper, we present an arithmetic analogue of the proof by Fox,
for the group $\F_2^n$. The argument can also be extended to other
Abelian groups, but we restrict ourselves to $\F_2^n$ for simplicity.

\begin{theorem} 
\label{thm:main}
For all $\eps \in (0,1]$, there exists a $\delta \geq
1/\tower{O(\log((1/\eps))}$ such that for all $n \in \N$ and $N \defeq
2^n $, any subset $A \subseteq \F_2^n$ which is $\eps$-far from being
triangle-free, must contain at least $\delta N^2$ triangles.
\end{theorem}

We remark that the above result (for all groups) already follows from
a version of the removal lemma for directed cycles, using a reduction
by Kr\'{a}l, Serra and Vena \cite{KralSV09}. This was already
observed by Fox \cite{Fox11} (see also \cite{ConlonF12}). However, we
believe some of the Fourier analytic notions that come up in a direct
arithmetic proof might be of independent interest. Also, a direct
proof makes it somewhat more transparent how the argument partitions
the underlying group, which might be useful for further
improvements. We present a sketch of the proof below.

\paragraph{Known lower bounds for triangle removal.}
The original proof by Ruzsa and \Szemeredi used \Szemeredi's
regularity lemma \cite{S75}. However, the known lower bounds for the
regularity lemma and its variants (see \cite{Gowers97} and
\cite{ConlonF12}) imply that this approach necessarily obtains a bound
on $1/\delta$ which is at least $\tower{\poly(1/\eps)}$. Fox's
argument \cite{Fox11} manages to avoid using the regularity lemma
directly (although his proof still follows the same outline as the
proof of the regularity lemma), thus obtaining $1/\delta \leq
\tower{O(\log(1/\eps))}$.

In terms of lower bounds on $1/\delta$, it was shown by Alon \cite{Alon02} that one must have
$1/\delta \geq 2^{\Omega(\log^2(1/\eps))}$ for triangle removal in graphs. For the problem of
triangle removal in groups, in the case when the group is $\Z/p\Z$, Green~\cite{Green05} gives a similar lower bound of $1/\delta \geq 2^{\Omega(\log^2(1/\eps))}$. However, in the case of $\F_2^n$, the only known lower bounds are polynomial in $\eps$. Bhattacharyya and Xie \cite{BhatX10} show that for triangle removal in $\F_2^n$ one must have $1/\delta \geq
(1/\eps)^{8.487}$, which was later improved to $(1/\eps)^{13.239}$ by \cite{FuK14}.
These lower bounds remain quite far from the known upper bounds on $1/\delta$.

\subsection{Proof Sketch}
We will give a proof by contradiction. Let $G$ denote the group
$\F_2^n,$ and assume, to the contrary, that we have a set $A \subseteq
G$ that is $\eps$-far from being triangle-free, and has at most $\delta
|G|^2$ triangles. We work with a subset $A' \subseteq A$ that is the
union of a maximal collection of element disjoint triangles,
\emph{i.e.} no two triples representing triangles share a common
element. Since $A\backslash A'$ is triangle-free, and $A$ is
$\eps$-far from being triangle-free, $|A'| \ge \eps|G|.$ Define
$\eps_0 \defeq \nfrac{|A'|}{|G|}.$

In the rest of the sketch, we denote $A'$ by $A,$ and $\eps_0$ by
$\eps.$ The proof is based on a potential-increment argument. At every step in
the proof, we have a partition of $G$ into $T$ cosets of a subgroup $H
\preceq G$ where $T \defeq \abs{\nfrac{G}{H}}$, and hence an
induced partition of $A$ according to cosets of $H.$ We measure the
\emph{mean entropy} of the partition, defined as
$\Ex{g}{\ent \left(\frac{|A \cap (H+g)|}{|H|}\right)},$ where $\ent(x) \defeq
x\log x$ for $x \in (0,1]$ and $\ent(0)\defeq 0.$ Observe that the mean
entropy is always non-positive, and by convexity it is at least $\eps
\log \eps.$ Our main lemma proves that if $\delta$ is much smaller
than $\eps^3\cdot \nfrac{|G|^2}{T^2}$, then we can \emph{shatter} $A,$
\emph{i.e.}, the current partition of $A$ can be refined according to
cosets of $H' \preceq H,$ such that the mean entropy increases by
$\Omega(\eps),$ and $\abs{\nfrac{G}{H'}} \le 2^{T\cdot
  O(\epsilon^{-3})}.$ In essence, the size of the partition, and hence
the bound on $\delta,$ are one exponential larger at every step. Since
the mean entropy is always non-positive, this process must stop after
$O(\log {\eps}^{-1})$ rounds, giving the required bound on $\delta.$

In order to show that we can shatter $A$ if it has too few triangles,
it is convenient to equate $A$ with its indicator function $A: G \to
\{0,1\}$, and the number of triangles with the sum $\sum_{x+y+z=0}
A(x)A(y)A(z).$ Suppose we want to count the number of triangles
between two cosets of $H,$ viz.  $H+g_1$ and $H+g_2,$ and a third
coset of $H',$ $H' + g_3 + z,$ where $g_1 + g_2 + g_3 = 0.$ Assume $A$
has density $\eps$ on all the three cosets. As a thought experiment,
if $A$ was the constant function $\eps$ on $H'+g_3+z,$ then the
``number of triangles'' between the three cosets is given by
$\eps|H||H'|\Ex{g \in H}{\frac{|A\cap (H'+g+g_1)|}{|H'|}\cdot
  \frac{|A\cap (H'+g+g_2+z)|}{|H'|}}.$ 
If this is significantly smaller
than $\eps^3|H||H'|$, then a Markov argument implies that partitioning according to cosets of $H'$
\emph{shatters} A \ie a constant fraction of the cosets have density significantly smaller than $\eps$.
A defect version of Jensen's
inequality then implies that the mean entropy of the new partition is
larger by $\Omega(\eps)$.

Of course, $A$ is not necessarily a constant function. However, a very
similar argument works if all the non-zero Fourier coefficients of the
function $A$ on $H'+g_3+z$ are much smaller than $\eps$ in absolute
value -- we call such functions \emph{superregular}, in analogy
with the proof from~\cite{Fox11}. The last part of the proof is to
find a \emph{superregular decomposition} of $A$ on $H+g_3,$
\emph{i.e.}  approximating $A \cap (H+g_3)$ by a sum of superregular
functions. If $A$ is not superregular on $H+g_3,$ we pick a large
Fourier coefficient $\eta \in \widehat{H},$ partition $x \in H$
according to the value of $\langle x,\eta\rangle,$ and restrict
ourselves to the part with greater density. If this part is not
superregular, we repeat this procedure. Since the density on any part
is at most 1, this process must end with a superregular part. We
remove this set from $A$ and repeat the procedure until most of $A$ is
covered, to find the required superregular decomposition.

This completes the proof sketch of Theorem~\ref{thm:main}.

\paragraph{Analogies and differences with the proof in~\cite{Fox11}.}
The proof of our main theorem on triangle-removal in groups is
analogous to the proof of triangle-removal in graphs from the work of
Fox~\cite{Fox11}. Nevertheless, we need to give the appropriate
arithmetic analogues of the definitions and proofs. Though the proof
in this paper is self-contained, for the readers who are familiar with
the work of Fox~\cite{Fox11}, we point out the analogies and the
differences between the two proofs.

At every step in our proof, as in~\cite{Fox11}, we have a partition of
the underlying set, the group $G$ in our case. However, our partition
is structured, and consists of all the cosets of one fixed subgroup,
compared to an arbitrary partition in~\cite{Fox11}. Our definitions of
the potential function and shattering are similar, except that we use
the densities of the cosets, instead of the edge densities between
pairs of subsets used in~\cite{Fox11}. Our notion of regularity is
quite different from that in~\cite{Fox11}, and is based on Fourier
coefficients, similar to the one used in regularity lemmas for Abelian
groups~\cite{Green05}. The superregular decomposition that we find for
a given set is analogous to the collection of superregular tuples in a
graph, constructed in~\cite{Fox11}.

\section{Preliminaries}
Fix a positive integer $n$. Throughout the paper, we denote $G \defeq
\F_2^n$ and $N \defeq |G| = 2^n.$ The notation $H \preceq G$ denotes
that $H$ is a subgroup of $G.$ 
We use $\widehat{H}$ to denote the dual
group of $H$. Denote by $H^\perp$ coset group of $H$ in $G$, which we also use to denote a set of
coset representatives. In some cases which will be clear from the context, by abuse of notation, 
we use the common definition of $H^\perp =\{ y ~\vert~ \sum_{i=1}^n x_i \cdot y_i  = 0 \;\; \forall x \in
H\}$. Note that for $G = \F_2^n$, 
\[
\hat{H} ~\cong~ G/H^{\perp} ~\cong~ H \mcom
\]
and hence $H^{\perp}$ can also be thought of as the coset group $G/H$.

Given a set $A \subseteq G,$ three elements $x,y,z \in A$ are said to
form a \emph{triangle} if $x+y+z=0.$ $A$ is said to be $\eps$-far from
being triangle-free, if at least $\eps N$ elements need to be removed
from $A$ in order that it contains no triangles.

We abuse notation and use $A$ to denote both the set $A$, and also its
characteristic function $A: G \to \B.$ Given a subgroup $H \preceq G,$
and an element $g\in G,$ we define $A_H^{g} : H\rightarrow \{0,1\}$ as
\[ A_H^g(x) \defeq A(x+g). \]

Let $\E_{x \in H} [\cdot]$ denote the expectation when $x$
is drawn uniformly from $H.$ For a function $f : H \to \R,$ we define
its Fourier coefficients as follows:
For $\eta\in \widehat{H},$ define
\[ \widehat{f}(\eta) \defeq \Ex{x \in H}{f(x) \chi_{\eta}(x)}.\] 

Jensen's inequality states that if $\ent$ is a convex function,
$\eps_1,\ldots,\eps_s$ are nonnegative real numbers such that
$\sum_{i=1}^s \eps_i = 1,$ then, for any real numbers
$x_1,\ldots,x_s,$
\begin{equation}
\label{eq:jensen}
\eps_1\ent(x_1)+ \cdots +\eps_s\ent(x_s) \ge \ent(\eps_1 x_1 + \cdots +\eps_s
x_s).
\end{equation}
Jensen's inequality immediately implies the following simple lemma,
which will be useful later.
\begin{lemma}[Lemma 6,~\cite{Fox11}]
\label{lem:jensen-simple}
Let $\ent: \R_{\geq 0} \rightarrow \R$ be a convex function,
$\epsilon_1,\ldots,\epsilon_s$ and $x_1,\ldots,x_s$ be nonnegative
real numbers with $\sum_{i\in [s]} \epsilon_i = 1$. For $I\subseteq
[s]$, let $c=\sum_{i\in I} \epsilon_i,$ $u=\sum_{i\in I}
\epsilon_ix_i/c$, and $v=\sum_{i\in [s]\backslash I}
\epsilon_ix_i/(1-c)$. Then we have
$$
\sum_{i\in [s]} \epsilon_i \ent(x_i)\geq c\ent(u)+ (1-c)\ent(v). 
$$
\end{lemma}

\subsection{Entropy}
Our proof of the main theorem will be based on a potential-increment argument,
where our potential function will be the \emph{mean entropy of a partition}, as
defined below.

\begin{definition}[Entropy]
  Define the \emph{entropy function} $\ent : \R_{\geq 0} \to \R$ as $\ent(x)
  \defeq x\log x$, for $x \in \R_{>0}$, and $\ent(0)=0$.
\end{definition}
\begin{definition}[Mean Entropy of a Partition]
  Let $H^\prime \preceq H \preceq G,$ and $g \in G.$ Given a set $A
  \subseteq G,$ if we partition the coset $H+g$ as cosets of
  $H^\prime,$ define the mean entropy of this partition as follows
\[ \ent_A(H+g,H^\prime) \defeq \E_{g_1\in H+g}{\ent\left(\frac{|A\cap
    (H^\prime + g_1)|}{|H^\prime|}\right)} = \E_{g_1\in H}{\ent\left(\frac{|A\cap
    (H^\prime + g + g_1)|}{|H^\prime|}\right)}.
\]
Analogously, the mean entropy of partitioning $A$ on $G$ according to
cosets of $H\preceq G$ is defined as
\begin{equation*}
  \ent_A(G,H) \defeq \E_{g_1 \in G}{\ent \left(\frac{|A\cap
        (H + g_1)|}{|H|}\right)}.
\end{equation*}

\end{definition}

\begin{remark}\label{remark:cosetentropy}
  Assume $A \subseteq G$ has been partitioned according to cosets of a
  subgroup $H\preceq G$. For a refinement to this partition according
  to $H'\preceq H$ we have
\begin{equation*}
  \ent_A(G,H')= \E_{g\in G}{\ent \left(\frac{|A\cap
        (H^\prime + g)|}{|H^\prime|}\right)} = \E_{g
    \in G} \ent_A(H+g,H').
\end{equation*}
\end{remark}

In order to quantify the increase in the mean entropy because of a
shattering, we need the following defect version of Jensen's
inequality for the entropy function. 
We include a proof for completeness.
\begin{lemma}[Defect inequality for entropy. \cite{Fox11}, Lemma 7]
\label{lem:defectentropy}
Let $\epsilon_1,\ldots,\epsilon_s,$ and $x_1,\ldots,x_s,$ be
nonnegative real numbers with $\sum_{i\in [s]} \epsilon_i=1$, and
$a=\sum_{i\in [s]} \epsilon_ix_i$. Suppose $\beta< 1,$ and
$I\subseteq [s]$ is such that $x_i\leq \beta a$ for all $i\in I.$ Let
$c=\sum_{i\in I} \epsilon_i$. Then,
\[\sum_{i\in [s]} \epsilon_i \ent(x_i) \geq \ent(a) + (1-\beta + \ent(\beta))
ca.\]
\end{lemma}
\begin{proof}
We know that $c<1$ since otherwise, $a=\sum_{1\leq i\leq s}
\epsilon_ix_i= \sum_{i\in I} \epsilon_ix_i\leq \beta a < a,$ a
contradiction. The cases when $a$ or $c$ are equal to $0$ follow
immediately from Jensen's inequality (Equation~\eqref{eq:jensen}),
therefore we may assume that $a,c\ne 0$.

Letting $u\defeq\frac{1}{c}\sum_{i\in I}\epsilon_i x_i$ and
$v\defeq\frac{1}{1-c} \sum_{i\not\in I}\epsilon_i x_i,$ we have,
\begin{align*}
\sum_{1\leq i \leq s} \epsilon_i \ent (x_i) 
&~\geq~ c \ent (u) + (1-c)\ent (v) \\
&~=~ \ent(a)+ ca\ent(u/a) + (1-c)a\ent(v/a)\\ 
&~\geq~ \ent(a)+ ca \ent(u/a) + ca(1-u/a) \\
&~=~ \ent(a)+ \big(\ent(u/a)+1-u/a\big)ca\\
&~\geq~ \ent(a)+ (1-\beta+\ent(\beta))ca,
\end{align*}
where the first inequality is from Lemma~\ref{lem:jensen-simple}, the first
equality follows from the definition of the entropy function $\ent$, the
second inequality follows from the fact that $\ent(\frac{v}{a})=
\ent(\frac{1-uc/a}{1-c})=\ent(1+\frac{(1-u/a)c}{1-c})>
\frac{(1-u/a)c}{1-c}$, and the last inequality follows from the fact
that $u/a \leq \beta$ and that $\ent(x)+1-x$ is a decreasing function on
the interval $[0,1]$.
\end{proof}

The following lemma shows how the mean entropy of a partition compares
to that of its refinement.

\begin{lemma}[Entropy - Basic inequalities]
\label{lem:entropy}
Let $H^\prime \preceq H \preceq G,$ and $A \subseteq G.$
\begin{enumerate}
\item $0 \ge \ent_A(G,H) \geq \ent_A(G,G)= \ent\left(\frac{|A|}{|G|}\right).$
\item For every $g \in G$, we have $\ent_A(H+g,H') \geq
  \ent_A(H+g,H)$, and therefore $\ent_A(G,H^\prime) \geq \ent_A(G,H).$
\end{enumerate}
\end{lemma}
\begin{proof}
Both parts follow from convexity of $\ent,$ and Jensen's inequality (Equation~\eqref{eq:jensen}).
\end{proof}

In Lemma~\ref{lem:shattering-entropy-increase}, we will show how
shattering a partition can substantially increase its mean entropy,
which will allow us to use the mean entropy as a potential function.

\section{Super-regularity}
In this section, we define a notion of regularity, and show how one
can approximate any set $A$ as a sum of regular parts.
\begin{definition}[$\rho$-Superregularity]
Let $H \preceq G,$ and $g \in G$. Given a function $f: H+g \to \R,$
say that $f$ is $\rho$-superregular on $H+g$ if, for every $\eta \in
\widehat{H},$ $\eta \neq 0,$ we have, $\abs{\widehat{f_H^{g}}(\eta)}
\le \rho\cdot \widehat{f_H^{g}}(0),$ where the function $f_H^{g} : H
\to \R$ is
defined as $f_H^{g}(x) \defeq f(x+g)$ for all $x \in H.$
\end{definition}

Similar notions of regularity have been used in the proof of a
Szemer\'{e}di type regularity lemma for $\F_2^n$~\cite{Green05}, and
are well-studied as notions of pseudorandomness for subsets of abelian
groups (\emph{e.g.} see \cite{CG92, Gowers98}).

Given a set $A,$ the following lemma identifies a subset of $A$ that
is superregular.
\begin{lemma}[Finding Superregular Parts]
\label{lem:one-regular-part}
  Let $H$ be a subgroup of $G.$ Given an element $g \in G,$ a desired
  regularity parameter $\rho \in (0,1],$ a density parameter $d > 0,$
  and a set $A \subseteq H + g$ such that $|A| \ge d|H|,$ we can find
  a triple $(A_1,H_1,z_1)$ such that:
\begin{enumerate}
\item $H_1$ is a subgroup of $H$ satisfying $\left| \nicefrac{H}{H_1} \right| \le
  2^{\log_{(1+\rho)} (1/d)}.$
\item $z_1 \in H \cap H_1^{\bot}.$
\item $A_1 \subseteq A,$ $A_1 \subseteq H_1 + z_1+g,$ and $|A_1| \ge d|H_1|.$
\item The indicator function of $A_1$ restricted to $H_1 + z_1 + g$ is
  $\rho$-superregular on $H_1+z_1+g.$
\end{enumerate}
\end{lemma}
\begin{proof}
We give an iterative procedure to find $A_1,H_1$ and $z_1.$ Initialize
$A_1 \defeq A,$ $H_1 \defeq H$ and $z_1 = 0.$ Observe that $z_1 \in H
\cap H_1^{\bot},$ and $\abs{A_1} \ge d\abs{H_1}.$

If the indicator function of $A_1$ restricted to $H_1+g + z_1$ is
${\rho}$-superregular on $H_1+g+z_1$, we are done and we can return
$(A_1,H_1,z_1).$

Otherwise, we must have $\eta \in \widehat{H_1}\backslash\{0\}$ for which 
$|\widehat{{A_1}^{g+z_1}_{H_1}}(\eta)| \ge \rho
\cdot \widehat{{A_1}_{H_1}^{g+z_1}}(0).$ Define $H_2\defeq H_1 \cap
\{h \ |\ \langle h, \eta \rangle = 0 \}$, so that $\abs{\nfrac{H_1}{H_2}}
= 2.$

If $|A_1 \cap (H_2 + g + z_1)| \ge |A_1 \cap (H_2 + g + z_1 + \eta)|,$
let $z_2 \defeq z_1.$ Otherwise, let $z_2 \defeq z_1 + \eta.$ Note
that $z_2 \in H \cap H_2^{\bot}.$

Defining $A_2 \defeq A_1 \cap (H_2 + g+ z_2),$ we have 
\begin{align*}
|\widehat{{A_1}^{g+z_1}_{H_1}}(\eta)|&= \frac{|A_2|- |A_1 \cap (H_2 + g + z_2 + \eta)|}{|H_1|}
\\ &\geq \rho  \widehat{{A_1}_{H_1}^{g+z_1}}(0) = \rho \cdot \frac{|A_1|}{|H_1|},
\end{align*}
moreover $\frac{|A_2|+|A_1 \cap (H_2 + g + z_2 + \eta)||}{|H_1|}=\frac{|A_1|}{|H_1|}$. Consequently,
$$
2\cdot \frac{|A_2|}{|H_1|} \geq (1+\rho)\cdot \frac{|A_1|}{|H_1|}.
$$
Thus since $\abs{\nfrac{H_1}{H_2}}
= 2$, we have that
$\frac{|A_2|}{|H_2|} \ge (1+\rho)\cdot \frac{|A_1|}{|H_1|},$
and $|A_2| \ge d|H_2|$ in particular.

Now, we let $H_1 \defeq H_2,$ $A_1 \defeq A_2$ and $z_1 \defeq z_2,$
and repeat the whole procedure. At every step, the triple
$(A_1,H_1,z_1)$ satisfies properties 2 and 3. 

Since the density of $A_1,$ \emph{i.e.} $\frac{|A_1|}{|H_1|},$ is
increasing by a factor of $(1+\rho)$ at every step, and can be at most
1, in $\log_{(1+\rho)} (1/d)$ iterations we must find a triple such
that the restriction of $A_1$ to $H_1+z_1+g$ is $\rho$-superregular on
$H_1 + z_1 + g$. Thus, the final triple satisfies property~4. Finally,
property~1 also holds since at every step, $|\nfrac{H}{H_1}|$
increases by a factor of 2.
\end{proof}

Repeatedly applying the above lemma gives us the following corollary,
which allows us to decompose a set into superregular parts.
\begin{corollary}[Superregular Decomposition]
\label{cor:regularizing}
Let $H$ be a subgroup of $G.$ Given an element $g \in G,$ a desired
regularity parameter $\rho \in (0,1],$ a density parameter $d \in
(0,1],$ and a set $A \subseteq H + g,$ we can find a positive integer
$t,$ and a collection of $t$ triples $(A_i,H_i,z_i),$ $i =
1,\ldots,t,$ such that $\{A_i\}_{i=1}^t$ are disjoint subsets of $A$
satisfying $\left|A \backslash (A_1 \cup \cdots \cup A_t) \right| \le d
|H|,$ and for every $i=1,\ldots,t$:
\begin{enumerate}
\item \label{item:cor:several-regular-parts:1} $H_i$ is a subgroup of
  $H$ satisfying $\left| \nicefrac{H}{H_i} \right| \le
  2^{\log_{(1+\rho)} (1/d)}.$
\item $z_i \in H \cap H_i^{\bot}.$ 
\item $A_i \subseteq A,$ $A_i \subseteq H_i + z_i+g,$ and $\abs{A_i}
  \ge d\abs{H_i}$.
\item The indicator function of $A_i$ restricted
  to $H_i + z_i + g$ is $\rho$-superregular on $H_i + z_i +g$.
\end{enumerate}

\end{corollary}
\begin{proof}
  We build the collection of triples one at a time. At the
  $i^\textrm{th}$ step, let $B \defeq A \backslash (A_1 \cup
    \cdots \cup A_{i-1}).$ If $|B| \le d|H|,$ we can return the
  triples found till step $(i-1).$

  So we can assume $|B| > d |H|,$ and apply
  Lemma~\ref{lem:one-regular-part} with subgroup $H,$ element $g,$
  regularity parameter $\rho,$ density parameter $d,$ and subset $B
  \subseteq H+g,$ to find $(A_i, H_i, z_i)$ satisfying the required
  properties 1-4. Moreover, since $A_i \subseteq A \backslash (A_1 \cup
    \ldots \cup A_{i-1}),$ $A_i$ is disjoint from $A_1,\ldots,A_{i-1}.$

We add the triple $(A_i, H_i, z_i)$ to the collection and iterate.
\end{proof}

\section{Shattering}
In this section, we prove that if a set $A$ contains few triangles
from three cosets, then we can partition at least one of the cosets
into parts with significantly varying densities. In order to state the 
concerned lemma formally, we need to define \emph{shattering}. The
following definition is similar to that used in Fox's proof of the
graph removal lemma~\cite{Fox11}.
\begin{definition}[Shattering]
  Let $H$ be a subgroup of $G.$ Given $g \in G$ and a set $A \subseteq
  G,$ define $d\defeq\frac{|A \cap (H+g)|}{|H|}$. Given a subgroup
  $H^\prime \preceq H,$ and parameters $\alpha,\beta \in (0,1],$ we
  say that $H^\prime$ ($\alpha,\beta,k$)-shatters $A$ on $H+g$ if:
\begin{enumerate}
\item $|\nicefrac{H}{H'}| \leq 2^k$, and 
\item $\Pr_{g'\in H} \left[ \left|A\cap (H^\prime + g +
      g^\prime)\right| \leq \beta d \abs{H^\prime} \right] \geq \alpha.$
\end{enumerate}
Namely, partitioning $H+g$ according to cosets of $H'$ results in
significant variation in the density of $A$.
\end{definition}

Once we find a shattering, the defect version of Jensen's inequality
allows us to prove that the mean entropy of the partition increases.
\begin{lemma}[Entropy Increment]
\label{lem:shattering-entropy-increase}
Let $H^\prime \preceq H \preceq G,$ and $A \subseteq G.$ Let $g\in
G$. If $H'$ ($\alpha,\beta,k$)-shatters $A$ on $H+g$ then
\begin{equation*}
  \ent_A(H+g,H^\prime) \geq \ent_A(H+g,H)+ (1-\beta+\ent(\beta))\alpha
  \cdot \frac{|A \cap (H+g)|}{|H|}.
\end{equation*}
\end{lemma}
\begin{proof}
Follows from the definitions of $\ent_A$ and shattering, and applying
Lemma~\ref{lem:defectentropy}. 
\end{proof}

Now, we can state the main lemma of this section.
\begin{lemma}[Shattering Lemma]
\label{lem:shattering}
Let $H\preceq G,$ and $g_1,g_2,g_3 \in G$ such that
$g_1+g_2+g_3=0$. We are given a set $A \subseteq G$ with densities
$d_1, d_2,$ and $d_3$ on $H+g_1$, $H+g_2,$ and $H+g_3$
respectively. Then, either $A$ contains at least $\frac{1}{8}d_1d_2d_3 |H|^2$
triangles, or else, there is a subgroup $H'\preceq H$ such that 
$H'$ $\left(\nfrac{1}{20}, \nfrac{3}{4}, \log_{1+\rho}
  \left(\nicefrac{2}{d_1}\right)\right)$-shatters $A$ on at least one
of $H+g_2$ or $H+g_3$, where $\rho=\frac{d_2d_3}{4}$.
\end{lemma}
We first prove some necessary lemmas, and then give a proof of the
above lemma at the end of this section. The following lemma allows us
to count the number of triangles between three different sets.
\begin{lemma}[Triangle Counting]
\label{lem:triangle-counting}
Let $H^\prime \preceq H \preceq G.$ We are given $g_1,g_2,g_3 \in G$
such that $g_1 + g_2 + g_3 = 0,$ and $z_1 \in H.$ For any sets $A,B,C
\subseteq G,$ the number of triangles between $A \cap
(H^\prime+g_1+z_1), B \cap (H+g_2), $ and $C \cap (H+g_3)$ is given
by,
\[|H| |H^\prime| \sum_{\alpha \in \widehat{H^\prime}, \eta \in
    \widehat{H^\prime}^{\bot} \cap \widehat{H}}
  \widehat{A_{H'}^{g_1+z_1}}(\alpha) \widehat{B_H^{g_2}}(\alpha + \eta)
  \widehat{C_{H}^{g_3}}(\alpha+\eta) \chi_{\alpha + \eta}(z_1).\]
\end{lemma}
\begin{proof}
We first observe that for any $x_1 \in H^\prime$, and $x_2,x_3 \in H,$
three elements $x_1 + g_1 + z_1\in A,\ x_2 + g_2 \in B,$ and $x_3 + g_3 \in C,$ form a triangle iff $x_1 + x_2 + x_3 = z_1,$ since, 
\[(x_1 + g_1 + z_1) + (x_2 + g_2) + (x_3 + g_3) = x_1 + x_2 + x_3 +
z_1.\] 
Thus, in order to count the triangles in the three cosets, we count
all such triples $x_1,x_2,x_3$. The number of triangles equals:
\begin{align*}
& \sum_{\substack{x_1 \in
H^\prime,x_2,x_3 \in H \\
x_1 + x_2 + x_3 = z_1}} A(x_1 + g_1 + z_1)
B(x_2 + g_2) C(x_3 + g_3) \\
  & = \sum_{\substack{x_1 \in H^\prime,x_2,x_3 \in H \\ x_1 + x_2 + x_3
= z_1}} A_{H'}^{g_1+z_1}(x_1) B_H^{g_2}(x_2) C_{H}^{g_3}(x_3) \\
  & = \sum_{\substack{x_1 \in H^\prime,x_2,x_3 \in H \\ x_1 + x_2 + x_3
= z_1}} \sum_{\alpha \in \widehat{H^\prime}, \beta,\gamma \in
\widehat{H}} \widehat{A_{H'}^{g_1+z_1}}(\alpha)
\widehat{B_H^{g_2}}(\beta) \widehat{C_{H}^{g_3}}(\gamma)
\chi_{\alpha}(x_1) \chi_{\beta}(x_2) \chi_{\gamma}(x_3)\\
  & = \sum_{x_1 \in H', x_2 \in H} \sum_{\alpha \in
\widehat{H^\prime}, \beta,\gamma \in \widehat{H}}
\widehat{A_{H'}^{g_1+z_1}}(\alpha) \widehat{B_H^{g_2}}(\beta)
\widehat{C_{H}^{g_3}}(\gamma) \chi_{\alpha}(x_1) \chi_{\beta}(x_2) \chi_{\gamma}(x_1+x_2+z_1)\\
  & = |H| |H^\prime| \sum_{\alpha \in \widehat{H^\prime}, \beta,\gamma
\in \widehat{H}} \widehat{A_{H'}^{g_1+z_1}}(\alpha)
\widehat{B_H^{g_2}}(\beta) \widehat{C_{H}^{g_3}}(\gamma)
\Ex{\substack{x_1 \in H' \\ x_2 \in H}} {\chi_{\alpha+\gamma}(x_1)
\chi_{\beta+\gamma}(x_2) \chi_{\gamma}(z_1)} \\
  & = |H| |H^\prime| \sum_{\alpha \in \widehat{H^\prime}, \beta \in
\widehat{H}} \widehat{A_{H'}^{g_1+z_1}}(\alpha)
\widehat{B_H^{g_2}}(\beta) \widehat{C_{H}^{g_3}}(\beta) \Ex{x_1 \in
H'} {\chi_{\alpha+\beta}(x_1) \chi_{\beta}(z_1)} \\
  & = |H| |H^\prime| \sum_{\alpha \in \widehat{H^\prime}, \beta \in
\widehat{H}, \alpha + \beta \in \widehat{H^\prime}^{\bot} }
\widehat{A_{H'}^{g_1+z_1}}(\alpha) \widehat{B_H^{g_2}}(\beta)
\widehat{C_{H}^{g_3}}(\beta) \chi_{\beta}(z_1) \\
  & = |H| |H^\prime| \sum_{\alpha \in \widehat{H^\prime}, \eta \in
\widehat{H^\prime}^\bot \cap \widehat{H}, }
\widehat{A_{H'}^{g_1+z_1}}(\alpha) \widehat{B_H^{g_2}}(\alpha+\eta)
\widehat{C_{H}^{g_3}}(\alpha + \eta) \chi_{\alpha + \eta}(z_1),
\end{align*}
where the last equality follows by viewing $\widehat{H'}$ as a subgroup of 
$\widehat{H}$.
\end{proof}

Next, we use the above lemma to show that if the number of triangles
between three sets is small, then we can shatter at least one of them.
\begin{lemma}[Shattering with a Superregular Part]
\label{lem:shattering-superregular}
Let $H^\prime \preceq H \preceq G.$ We have $g_1,g_2,g_3 \in G$ such
that $g_1 + g_2 + g_3 = 0,$ and $z_1 \in H.$ Suppose we are given three
sets $A,B,C \subseteq G,$ such that
\[\frac{|A \cap (H'+g_1+z_1)|}{|H'|} = d_1,\quad \frac{|B \cap
(H+g_2)|}{|H|} = d_2,\quad \frac{|C \cap (H+g_3)|}{|H|} = d_3. \]
Also assume that the indicator function of $A$ restricted to
$H'+g_1+z_1$ is $\frac{d_2d_3}{4}$-superregular on $H'+g_1+z_1.$

Then, either there are at least $\frac{1}{4} d_1 d_2 d_3 |H| |H'|$
triangles between $A \cap (H^\prime+g_1+z_1), B \cap (H+g_2), $ and $C
\cap (H+g_3)$, or else, $H^\prime$
$(\nfrac{1}{20},\nfrac{3}{4},\log_2 |\nicefrac{H}{H'}|)$-shatters
either $B$ on $H+g_2,$ or $C$ on $H+g_3.$
\end{lemma}
\begin{proof}
We first use Lemma~\ref{lem:triangle-counting} to count the triangles
between $A \cap (H^\prime+g_1+z_1), B \cap (H+g_2), $ and $C \cap
(H+g_3).$
\begin{align*}\widehat{H^\prime}^
  \textrm{No. of triangles } & = |H| |H^\prime| \sum_{\alpha \in
\widehat{H^\prime}, \eta \in \widehat{H^\prime}^{\bot} \cap \widehat{H}}
\widehat{A_{H'}^{g_1+z_1}}(\alpha) \widehat{B_H^{g_2}}(\alpha +
\eta) \widehat{C_H^{g_3}}(\alpha + \eta) \chi_{\alpha +
\eta}(z_1). \\
  & = |H||H^\prime| \left(\sum_{\eta \in \widehat{H^\prime}^{\bot} \cap
\widehat{H}} \widehat{A_{H'}^{g_1+z_1}}(0) \widehat{B_H^{g_2}}(
\eta) \widehat{C_H^{g_3}}(\eta) \chi_{\eta}(z_1) \right.\\
& \qquad \qquad + \left.  \sum_{\alpha \in \widehat{H^\prime}, \eta
\in \widehat{H^\prime}^{\bot} \cap \widehat{H}, \alpha \neq 0}
\widehat{A_{H'}^{g_1+z_1}}(\alpha) \widehat{B_H^{g_2}}(\alpha +
\eta) \widehat{C_H^{g_3}}(\alpha + \eta) \chi_{\alpha + \eta}(z_1)
\right) \\
  & \ge |H||H^\prime| \left(d_1\sum_{\eta \in \widehat{H^\prime}^{\bot} \cap
\widehat{H}} \widehat{B_H^{g_2}}(\eta) \widehat{C_H^{g_3}}( \eta)
\chi_{\eta}(z_1) \right.\\
& \qquad \qquad - \frac{d_1 d_2 d_3}{4} \left.  \sum_{\alpha \in
\widehat{H^\prime}, \eta \in \widehat{H^\prime}^{\bot} \cap \widehat{H},
\alpha \neq 0} \widehat{B_H^{g_2}}(\alpha + \eta)
\widehat{C_H^{g_3}}(\alpha + \eta) \chi_{\alpha + \eta}(z_1) \right),
\end{align*}
where the last inequality uses the fact that the indicator function of
$A$ restricted to $H'+g_1+z_1$ is $\frac{d_2 d_3}{4}$-superregular on
$H^\prime + g_1 +z_1,$ and that $\widehat{A_{H'}^{g_1+z_1}}(0) =
d_1.$

Using the Cauchy-Schwarz inequality, we get that the second term in the bracket is
at least $-\frac{1}{4} d_1 d_2 d_3 \sqrt{d_2 d_3} \ge -\frac{1}{4} d_1 d_2 d_3$. Thus, if we have fewer than
$\frac{1}{4}d_1 d_2 d_3 |H||H^\prime|$ triangles, we must have that
the first term in the bracket is at most $\frac{1}{2}d_1 d_2 d_3.$
This implies that,
\[\sum_{\eta \in \widehat{H^\prime}^{\bot} \cap \widehat{H}}
\widehat{B_H^{g_2}}(\eta) \widehat{C_H^{g_3}}( \eta)
\chi_{\eta}(z_1) \le \frac{d_2 d_3}{2}.\]

We prove the following lemma, which allows us to deduce that $H'$
shatters either $B$ on $H+g_2,$ or $C$ on $H+g_3.$ A proof has been
included later in the section.
\begin{lemma}[Fourier shattering]
\label{lem:fouriershattering}
Let $H^\prime \preceq H \preceq G.$ Given two functions $f$ and $g$
from $H$ to $\R_{\ge 0}$ that satisfy
\begin{equation}\label{eq:fouriershattering}
\sum_{\eta \in \widehat{H^\prime}^{\bot} \cap \widehat{H}} \widehat{f}(\eta)
\widehat{g}( \eta) \chi_{\eta}(z_1) \le \frac{ d_1 d_2}{2},
\end{equation}
for some
positive $d_1, d_2$ and $z_1 \in H ;$ define the function $\bar{f}(v)
\defeq \Ex{y \in H^\prime} {f(v+y)}$ and $\bar{g}(v) \defeq \Ex{y \in
  H^\prime} {g(v+y)},$ where $v \in H.$ Then,
either $\Prob{v}{\bar{f}(v) < \frac{3}{4}d_1} > \frac{1}{20},$ or
$\Prob{v}{\bar{g}(v) < \frac{3}{4}d_2} > \frac{1}{20}.$
\end{lemma}
Assuming this lemma, and applying it to the functions $B_H^{g_2}$ and
$C_H^{g_3},$ we deduce that for the functions $\bar{B_H^{g_2}}(v)
\defeq \Ex{y \in H^\prime} {B(v+y+g_2)} = \frac{|B \cap
  (H'+g_2+v)|}{|H'|}$ and $\bar{C_H^{g_3}}(v) \defeq \Ex{y \in
  H^\prime} {C(v+y+g_3)} = \frac{|C \cap (H'+g_3+v)|}{|H'|},$ for $v \in H$, 

we have either
$\Prob{v}{\bar{B_H^{g_2}}(v) < \frac{3}{4}d_2} > \frac{1}{20},$ or
$\Prob{v}{\bar{C_H^{g_3}}(v) < \frac{3}{4}d_3} > \frac{1}{20}.$ This
means that $H^\prime$
$(\nfrac{1}{20},\nfrac{3}{4},\log |\nfrac{H}{H'}|)$-shatters either
$B$ on $H+g_2$ or $C$ on $H+g_3.$
\end{proof}

We are now ready to give a proof of the main lemma.
\begin{proof}(\emph{of Lemma~\ref{lem:shattering}}).
Let $\rho \defeq \frac{d_2d_3}{4}$. Apply Corollary~\ref{cor:regularizing}
with subgroup $H,$ element $g_1,$ regularity parameter $\rho,$ density
parameter $\nfrac{d_1}{2},$ and the set $A \cap (H+g_1),$ to partition
$A \cap (H+g_1)$ into $\rho$-superregular parts. Let
$\{(A_i,H_i,z_i)\}_{i=1}^t$ be the triples returned by
Corollary~\ref{cor:regularizing} satisfying $\left| \nicefrac{H}{H_i}
\right| \le 2^{\log_{(1+\rho)} (2/d_1)},$ $A_i \subseteq H_i + z_i+g_1,$
$\abs{A_i} \ge \nfrac{d_1}{2}\abs{H_i},$ and that the indicator
function of $A_i$ restricted to $H_i + z_i + g_1$ is
$\rho$-superregular on $H_i + z_i + g_1$.

Fix a triple $(A_i,H_i,z_i).$ If $H_i$ $(\nfrac{1}{20}$,
$\nfrac{3}{4}$, $\log_2 |\nicefrac{H}{H_i}|)$-shatters $A$ on $H+g_2$
or $H+g_3,$ it also $(\nfrac{1}{20}, \nfrac{3}{4} , \log_{1+\rho}
\nfrac{2}{d_1})$-shatters it and we are done. Otherwise, applying
Lemma~\ref{lem:shattering-superregular} to the sets $A_i, A \cap
(H+g_2),$ and $A \cap (H+g_3),$ on the cosets $H_i+z_i+g_1,H+g_2,$ and
$H + g_3,$ respectively, there must be at least
$\frac{1}{4}\frac{|A_i|}{|H_i|}d_2d_3|H||H_i| =
\frac{1}{4}d_2d_3|A_i||H|$ triangles between $A_i, A \cap (H+g_2),$
and $A \cap (H+g_3).$

Repeating the above argument for every triple, assume we do not find a
subset $H'$ that $(\nfrac{1}{20}, \nfrac{3}{4}, \log_{1+\rho}
\nicefrac{2}{d_1})$-shatters $A$ on at least one of $H+g_2$ or
$H+g_3.$ Then, since the sets $\{A_i\}_i$ are disjoint and for all
$i,$ $A_i \subseteq A$ , the total number of triangles between $A \cap
(H+g_1), A \cap (H+g_2),$ and $A \cap (H+g_3)$ is at least
\begin{equation*}
  \sum_{i=1}^t
  \frac{1}{4}d_2d_3|A_i||H| = 
  \frac{1}{4}d_2d_3|H| \sum_{i=1}^t|A_i| \ge   \frac{1}{8}d_1d_2d_3 |H|^2,
\end{equation*}
where the last inequality follows because $\{A_i\}_{i=1}^t$ are
disjoint subsets satisfying $|(A \cap (H+g_1)) \backslash (A_1 \cup
\cdots \cup A_t)| \le \frac{d_1}{2}|H|,$ and hence $\sum_i |A_i| \ge
|A \cap (H+g_1)| - \frac{d_1}{2}|H| \ge \frac{d_1}{2}|H|.$ This
completes the proof.
\end{proof}

We now give a proof of Lemma~\ref{lem:fouriershattering}.
\begin{proof}(\emph{of Lemma~\ref{lem:fouriershattering}}).
We can simplify the left side of (\ref{eq:fouriershattering}) as follows:
\begin{align*}
\sum_{\eta \in {H^\prime}^{\bot} \cap \widehat{H}}
\widehat{f}(\eta) \widehat{g}( \eta) \chi_{\eta}(z_1) & =
\frac{1}{|H|^2}\sum_{x_1,x_2 \in H}{\sum_{\eta \in {H^\prime}^{\bot}
\cap \widehat{H}} f(x_1) \chi_{\eta}(x_1) g(x_2)\chi_{\eta}(x_2)
\chi_{\eta}(z_1)} \\
& = \frac{1}{|H||H^\prime|}\sum_{\substack{x_1,x_2 \in H \\ x_1 + x_2
+ z_1 \in {H^\prime}}} {f(x_1) g(x_2) }\\ & \hfill \textrm{(Since $x_1+x_2+z_1 \in H$ and $|H'||H'^{\bot}\cap \widehat{H}|  = |H|$)} \\
& = \frac{1}{|H||H^\prime|}\sum_{x_1 \in H, y_1 \in H^\prime}
{f(x_1) g(x_1 + y_1 + z_1) } \\
& = \frac{1}{|H||H^\prime|}\sum_{y_1,y_2 \in H^\prime, v \in H \cap
{H^\prime}^\bot } {f(v+y_2) g(v+y_2 + y_1 + z_1) } \\
& = \frac{|H^\prime|}{|H|}\sum_{v \in H \cap {H^\prime}^\bot } \Ex{y_2
\in H^\prime}{f(v+y_2) \left( \Ex{y _1 \in H^\prime}{g(v+y_2 + y_1 +
z_1)} \right) } \\
& = \frac{|H^\prime|}{|H|}\sum_{v \in H \cap {H^\prime}^\bot } \Ex{y_2
\in H^\prime}{f(v+y_2) \left( \Ex{y _1 \in H^\prime}{g(v + y_1 + z_1)} \right) } \\
& \hfill \textrm{(Since for $y_1$ uniform over $H',$ $y_1+y_2$ is also
  uniform over $H'$)} \\
& = \frac{|H ^\prime|}{|H|}\sum_{v \in H \cap
  {H^\prime}^\bot }\Ex{y_2 \in H^\prime}
{f(y_2+v)} \Ex{y_1 \in H^\prime} {g(y_1 + v + z_1) } \\
& = \Ex{v \in H \cap
  {H^\prime}^\bot }{\bar{f}(v) \bar{g}(v+z_1)}.
\end{align*}

Thus,
\[\Ex{v \in H \cap
  {H^\prime}^\bot }{\bar{f}(v) \bar{g}(v+z_1)}\le \frac{ d_1 d_2}{2}.\]
We now claim that either $\Prob{v}{ \bar{f}(v) < \frac{3}{4} d_1} >
\frac{1}{20},$ or $\Prob{v}{ \bar{g}(v) < \frac{3}{4} d_2} >
\frac{1}{20}.$ This holds because otherwise, since
$\bar{f}(v),\bar{g}(v) \ge 0,$
\[\Ex{v \in H \cap {H^\prime}^\bot } {\bar{f}(v) \bar{g}(v+z_1)} \ge
\left(1 - 2\cdot\frac{1}{20}\right) \cdot \frac{3}{4} d_1 \cdot
\frac{3}{4} d_2 = \frac{81}{160} d_1 d_2 > \frac{1}{2} d_1 d_2,\]
which is a contradiction.
\end{proof}

\section{Proof of the main theorem}
The proof of Theorem~\ref{thm:main} will follow by repeated
applications of the following lemma. 
\begin{lemma}[Main Lemma]
\label{lem:main}
Let $A\subseteq G$ be a set that is a union of $\epsilon N$ disjoint
triangles. Suppose we have partitioned $G$ into cosets of a subgroup
$H,$ and let $T \defeq |\nicefrac{G}{H}|$. If $A$ contains less than
$\frac{\epsilon^3}{64T^2}N^2$ triangles (not necessarily disjoint),
then there is a subgroup $H^\prime \preceq H$ such that:
\begin{enumerate}
\item $\ent_A(G,H^\prime) \geq \ent_A(G,H) + \frac{\epsilon}{3600}.$
\item $|\nicefrac{G}{H'}| \leq c^T$, where
  $c=2^{2\log_{1+\epsilon^2/16} (\nicefrac{4}{\epsilon})}$.
\end{enumerate}
\end{lemma}
\begin{proof}
First, remove those elements from $A$ which belong to cosets of $H$ in which 
 $A$ has density less than $\nfrac{\epsilon}{2}$. Let $A'$ be
what is left from $A$. Notice that the number of elements removed in
this process is at most $\epsilon N/2$, and hence $A'$ contains at
least $\epsilon N/2$ disjoint triangles.

Let $g_1,g_2,g_3$, be a triangle in $A',$ \emph{i.e.},
$g_1+g_2+g_3=0$, and for $i =1,2,3,$ let $d_i$ be the density of $A'$
in the coset $H+g_i$. Note that $d_1,d_2,d_3\geq
\nicefrac{\epsilon}{2}.$ Since $A'$ contains at most
$\frac{\epsilon^3N^2}{64T^2}\leq \frac{d_1d_2d_3}{8} |H|^2$ triangles,
by Lemma~\ref{lem:shattering}, there is a subgroup $H_1 \preceq H$
such that $H_1$ ($\nfrac{1}{20},\nicefrac{3}{4},\log_{1+{\eps^2}/{16}}
\nicefrac{4}{\epsilon}$)-shatters $A'$ on at least one of $H+g_1$ or
$H+g_2.$

For each coset $H+g$ that can be shattered, identify any one subgroup
that ($\nfrac{1}{20},\nicefrac{3}{4},\log_{1+{\eps^2}/{16}}
\nicefrac{4}{\epsilon}$)-shatters $A'$ on $H+g$. Let $H'$ be the
intersection of all these subgroups.  For a coset $H+g$ that is not
shattered, we have $\ent_A(H+g,H') - \ent_A(H+g,H) \geq 0$
(Lemma~\ref{lem:entropy}, part 2). For every coset $H+g$ such that
$A'$ is shattered on $H+g$, observe that $A \cap (H+g) = A' \cap
(H+g),$ and hence Lemma~\ref{lem:shattering-entropy-increase} implies,
\begin{align*}
  \ent_A(H+g,H') - \ent_A(H+g,H) 
& =   \ent_{A'}(H+g,H') - \ent_{A'}(H+g,H) \\ 
& \geq
  \frac{1}{20}\left(1-\frac{3}{4}+\ent\left(\frac{3}{4}\right)\right) \cdot
  \frac{|A' \cap (H+g)|}{|H|} \\
& \geq
  \frac{1}{600}\frac{|A' \cap (H+g)|}{|H|}.
\end{align*}

Since $A'$ contains at least $\epsilon N/2$ disjoint triangles, and at
least one element from each of these triangles is contained in a 
coset that is shattered, at least $\epsilon N/6$ of the elements of 
$A'$ belong to a coset that has been shattered. Thus, averaging over 
all the cosets of $H$ (using Remark~\ref{remark:cosetentropy}),
\begin{equation*}
  \ent_A(G,H') - \ent_A(G,H) \geq \frac{1}{600}\cdot \frac{\epsilon}{6} =
  \frac{\epsilon}{3600}.
\end{equation*}

Note that
$
|\nicefrac{H}{H'}| \leq 2^{T\log_{1+\epsilon^2/16}
  (\nfrac{4}{\epsilon})},
$
and hence,
$$
|\nicefrac{G}{H'}| \leq T2^{T\log_{1+\epsilon^2/16}
  (\nfrac{4}{\epsilon})} \leq 2^{2T\log_{1+\epsilon^2/16}
  (\nfrac{4}{\epsilon})} =c^T.
$$
\end{proof}

Now we show how to deduce the main theorem using the above lemma.
\begin{proof}(\emph{of Theorem~\ref{thm:main}}).
Let $\delta$ be such that $\delta^{-1}$ is a tower of twos of height
$\Theta(\log \nfrac{1}{\eps})$ (the constant in $\Theta$ will be specified
later). Assume for contradiction that $A$ is $\eps$-far from being
triangle-free, and has less than $\delta N^2$ triangles. Thus, more
than $\epsilon N$ elements have to be removed from $A$ to make it
triangle-free. Consider a maximal set of disjoint triangles in $A,$ of
say $\eps_0 N$ triangles, and let $A'\subseteq A$ be the union of
these triangles. Thus, $|A'| = 3\eps_0 N.$ From now on, we will only
work with $A'$. Since $A \backslash A'$ must be triangle-free,
$3\eps_0 N \ge \eps N,$ \emph{i.e.}, $\eps_0 \ge \nfrac{\eps}{3}.$

Since $A$ has at most $\delta N$ triangles, $A'$ also has at most
$\delta N$ triangles. Now, we repeatedly apply Lemma~\ref{lem:main} in
order to find successively finer partitions of $G$ with increasing
mean entropies. For $i=0,$ start with the trivial partition according
to cosets of $H_0 \defeq G$ for which,
\begin{enumerate}
\item the number of parts is $T_0 \defeq|\nicefrac{G}{G}|=1$, and,
\item the mean entropy density of the partition is $\ent_{A'}(G,H_0) =
  3\epsilon_0 \log 3\epsilon_0$.
\end{enumerate}

At step $i,$ if $\delta N^2\leq \frac{\epsilon_0^3}{64 T_i^2} N^2,$
then we can apply Lemma~\ref{lem:main} to refine the partition
according to a subgroup $H_{i+1}\preceq H_{i}$, such that
$\ent_{A'}(G,H_{i+1}) \ge \ent_{A'}(G,H_i) + \frac{\eps_0}{3600}.$
Moreover $T_{i+1} \leq 2^{2 T_i \log_{1+\epsilon_0^2/16}
  (\nicefrac{4}{\epsilon_0})}.$ Since $\delta^{-1}$ is a tower of twos
of height $\Theta(\log \nfrac{1}{\epsilon}) = \Omega(\log
\nfrac{1}{\epsilon_0})$, we can pick the constant inside the $\Theta$
large enough so that the condition $\delta \le \frac{\epsilon_0^3}{64
  T_i^2}$ is satisfied for all $i \le t \defeq \left\lceil{12000 \log
    \nfrac{1}{3\eps_0}}\right\rceil.$ This implies $\ent_{A'}(G,H_t)
\ge 3\eps_0 \log 3\eps_0 + \frac{\eps_0}{3600}\cdot
 12000 \log \frac{1}{3\eps_0} > 0.$ However, we must always have
$\ent_{A'}(G,H_t) \leq 0$ (Lemma~\ref{lem:entropy}, part 1), and hence
this is a contradiction.
\end{proof}

\subsection*{Acknowledgements}
The authors would like to thank Oded Regev and anonymous
  reviewers for their valuable comments.

\bibliographystyle{amsalpha}
\bibliography{groups}


\end{document}